\documentclass[a4paper,10pt]{article}

\usepackage{graphicx}
\usepackage{amssymb}
\usepackage{amsthm}
\usepackage{amsmath}
\usepackage{enumitem}
\usepackage{graphicx}
\usepackage{psfrag}
\usepackage{txfonts}
\usepackage{bookman}

\usepackage[dvips]{hyperref}
\usepackage{breakurl}
\allowdisplaybreaks

\newtheorem{theorem}{Theorem}[section]

\newtheorem{corollary}[theorem]{Corollary}
\newtheorem{lemma}[theorem]{Lemma}

\theoremstyle{definition}

\numberwithin{equation}{section}
\numberwithin{figure}{section}

\newcommand{\C}{\mathbb C}
\newcommand{\R}{\mathbb R}

\newcommand{\edot}{\,\cdot\,}

\newcommand{\Ro}{\mathcal R}
\newcommand{\Fo}{\mathcal  F}
\newcommand{\Ho}{\mathcal H}
\newcommand{\Hi}{\mathrm  H}
\newcommand{\Io}{\mathcal  I}

\newcommand{\rmd}{\mathrm d}

\newcommand\abs[1]{\left\vert#1\right\vert}

\newcommand\set[1]{\left\{#1\right\}}

\newcommand{\win}{\mathbf n}
\newcommand{\ddim}{d}
\newcommand{\al}{\alpha}

\newcommand{\kk}{\mathbf k}
\newcommand{\ky}{\lambda}
\newcommand{\xx}{\mathbf x}
\newcommand{\yy}{y}
\newcommand{\uu}{\mathbf u}
\newcommand{\mm}{p}

\newcommand{\kl}[1]{\left(#1\right)}
\newcommand{\ekl}[1]{\left[#1\right]}
\newcommand{\mkl}[1]{\bigl(#1\bigr)}
\newcommand{\skl}[1]{(#1)}

\newcommand{\req}[1]{(\ref{eq:#1})}

\usepackage{geometry}
\usepackage{bbding}

\title{Exact reconstruction formulas for a Radon transform over cones}

\author{Markus Haltmeier}
\date{\today}

\begin{document}

\maketitle

\begin{abstract}
Inversion of  Radon transforms is the mathematical foundation of many modern tomographic imaging modalities. In this paper we study a conical Radon transform, which is important for computed tomography taking  Compton scattering into account. The conical Radon transform we study integrates a function in $\R^\ddim$ over all conical surfaces having vertices on a hyperplane and symmetry axis orthogonal to this plane.
As the  main result we
derive exact reconstruction formulas of the filtered back-projection type for inverting this transform.

\bigskip
\noindent
\textbf{Keywords.}
Radon transform, conical projections, computed tomography, inversion formula,
filtered back-projection.

\bigskip
\noindent
\textbf{AMS subject classifications.}
44A12, 45Q05, 92C55.
\end{abstract}

\section{Introduction}
\label{sec:intro}

Suppose that $f \colon \R^\ddim \to \R$,  with $\ddim \geq 2$,
is a smooth function supported in the half space
$\R^{\ddim-1} \times   \kl{0, \infty}$, and let
$\mm$ be some real number.
We  study the problem of reconstructing the function $f$ from the
integrals
\begin{equation} \label{eq:C-radon}
	\kl{\Ro^{(\mm)} f }\kl{\uu, \theta } \coloneqq
	\int_0^\infty \frac{1}{s^\mm} \int_{S^{\ddim-2}}
	f \kl{\uu + s \sin\kl{\theta} \win ,  s \cos\kl{\theta} }
	\kl{s \sin\kl{\theta}}^{\ddim-2} \rmd\win \rmd s
\end{equation}
for $\uu \in \R^{\ddim-1}$ and $ \theta \in \kl{0, \pi \slash 2}$. (Here $S^{\ddim-2}$ is the unit sphere in $\R^{d-1}$ and  $\rmd\win$ the surface measure on $S^{\ddim-2}$).
We call the function $\Ro^{(\mm)} f \colon \R^{\ddim-1}
\times (0,\pi\slash 2) \to \R$ the conical Radon transform of $f$. As illustrated in Figure~\ref{fig:cone}, $\mkl{\Ro^{(\mm)} f }\kl{\uu, \theta }$
is the integral of $f$ over the one sided conical surface $C\kl{\uu, \theta}$ having  vertex
$\kl{\uu, 0}$ on the plane $\R^{\ddim-1} \times\kl{0, \infty}$, symmetry axis
$\mathbf{e}_\ddim \coloneqq \kl{0, \dots, 0,1}$, and half opening angle $ \theta \in \kl{0,  \pi \slash 2}$.  The product $\kl{s \sin\kl{\theta}}^{\ddim-2} \rmd\win \rmd s$ is the standard surface measure on $C\kl{\uu, \theta}$, and $1\slash s^\mm$ is an additional  radial weight that can be adapted to a particular application at hand. For $\theta \in (0,\pi\slash 2)$, the function  $\mkl{\Ro^{(\mm)} f }\kl{\edot, \theta }$ may be considered  as a conical projection of $f$ onto $\R^{\ddim-1} \times \set{0}$.

\begin{psfrags}
\begin{figure}[tbh]
\centering
\psfrag{M}{$(\uu,0)$}
\psfrag{e}{vertex plane}
\psfrag{C}{$C\kl{\uu, \theta}$}
\psfrag{t}{\scriptsize $\theta$}
\psfrag{r}{\scriptsize $s \sin \kl{\theta}$}
\psfrag{n}{$\mathbf{e}_n$}
\psfrag{P}{\scriptsize $\kl{\uu + s \sin\kl{\theta} \win ,  s \cos\kl{\theta} }$}
\includegraphics[width=0.8\textwidth]{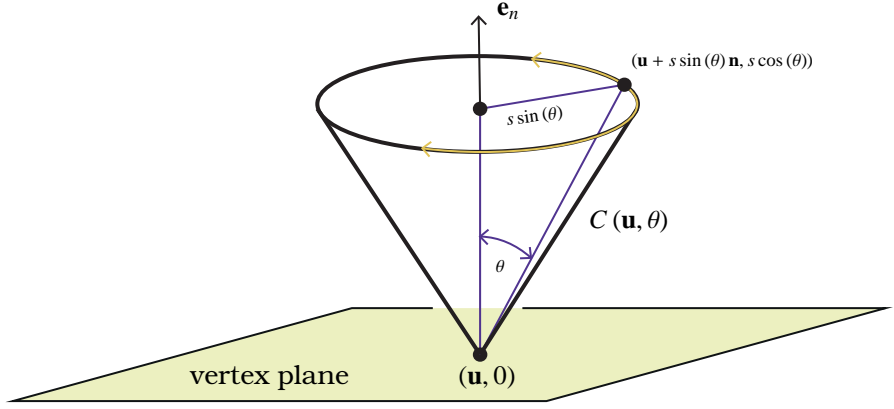}
\caption{The conical Radon transform\label{fig:cone} integrates a function with support in the upper half space
over one sided conical surfaces  $C\kl{\uu, \theta}$ centered  at $(\uu,0) \in \R^{\ddim-1} \times \kl{0, \infty}$  having  symmetry axis $\mathbf{e}_\ddim = \kl{0, \dots, 0,1}$
and half-angle $\theta \in \kl{0, \pi\slash 2}$.
Any point on $C\kl{\uu, \theta}$ can be written in the
form $\kl{\uu + s \sin\kl{\theta} \win ,  s \cos\kl{\theta}}$ with $\win \in S^{\ddim-2}$ and $s >0$. The $\ddim-1$ dimensional surface measure on $C\kl{\uu, \theta}$ is given by
$\kl{s \sin\kl{\theta}}^{\ddim-2} \rmd\win \rmd s$, with $\rmd\win$ denoting the
standard surface measure on $S^{d-2}$.}
\end{figure}
\end{psfrags}

Inversion of the conical Radon transform in three spatial dimension is important for computed tomography taking  Compton scattered photons into account \cite{BasZenGul98,CreBon94,Par00}.
In \cite{CreBon94,NguTruGra05} Fourier reconstruction formulas have been derived for the
cases  $\mm  \in \set{0,2}$.
For two spatial dimensions, $\Ro^{(\mm)}$ has been studied  with $\mm  \in \set{0,2}$ in \cite{BasZenGul97,TruNgu11},
where reconstruction formulas of the back-projection type have  been derived.  In dimensions $\ddim >3$, the conical Radon transform has, to the best of our knowledge,  not been studied so far. In this paper we study $\Ro^{(\mm)}$  for any $\ddim \geq 2$ and any $\mm \in \R$. We derive  explicit reconstruction formulas of the back-projection
type (see Theorem \ref{thm:fbp}) as well as a Fourier slice identity (see Theorem \ref{thm:fourierslice}) similar to the one of the classical Radon transform.

\subsection{Statement of  the main results}

Before we present our main results we introduce  some notation.
By $C^\infty_c \skl{\R^{\ddim-1} \times \kl{0, \infty}}$
we denote  the space of all functions defined on $\R^\ddim$, that are $C^\infty$ and have compact support in
$ \R^{\ddim-1} \times \kl{0, \infty}$.
Likewise $C^\infty\skl{\R^{\ddim-1} \times \kl{0, \pi\slash 2}}$ denotes the space of all infinitely smooth functions defined on
$\R^{\ddim-1} \times \kl{0, \pi\slash 2}$.
As can easily be seen, the conical Radon transform  defined by   \req{C-radon}
is well defined as an operator
$\Ro^{(\mm)} \colon C_c^\infty\mkl{\R^{\ddim-1} \times \skl{0, \infty}} \to C^\infty\mkl{\R^{\ddim-1} \times \skl{0, \pi\slash 2}}$.

Points in  $\R^\ddim$ will be written in the form
$\kl{\xx, \yy}$ with $\xx \in \R^{\ddim-1}$ and
$\yy \in \R$.
The Fourier transform of a function
$ f \colon \R^{\ddim-1} \times \R \to \C$ with respect to the first component is denoted by
$\kl{\Fo f}\kl{\kk, \yy} \coloneqq
	\int_{\R^{\ddim-1}} e^{-i\kk\xx}
	f \kl{\xx, \yy} \rmd \xx
	$
	for $\kl{\kk, \yy} \in \R^{\ddim-1} \times \R$.
The Hankel transform  of order $\kl{\ddim-3} \slash 2$ in the  second  argument is denoted by
$	\kl{\Ho_{(\ddim-3)/2} f } \kl{\xx, \ky}
	\coloneqq
	\int_0^\infty
	J_{\skl{\ddim-3}\slash 2} \kl{\yy \ky}
	f \kl{\xx, \yy}
    \yy \rmd \yy
    $ for
    $\kl{\xx, \ky} \in \R^{\ddim-1} \times \kl{0, \infty}$, where $J_{(\ddim-3)/2}$ is  the Bessel function of the first kind of order $\kl{\ddim-3} \slash 2$. Note that for $\ddim=2$, we have
$J_{-1/2}\kl{\yy}  =  \sqrt{2 \slash \pi \yy} \cos \kl{\yy}$ and hence $\Ho_{-1/2}$ is closely related  to the  cosine transform.

Similarly, we  denote by $ \Fo g$
the Fourier transform of a function $g \colon \R^{\ddim-1} \times \kl{0, \pi\slash 2} \to \C$ with respect to the first argument. Finally,
we denote by  $\Io^{(1-\ddim)}  g \colon \R^{\ddim-1} \times
\kl{0, \pi\slash 2} \to \C$ the  Riesz potential of $g$, defined by
\begin{equation} \label{eq:riesz}
\kl{\Fo \Io^{(1-\ddim)} g} \kl{\kk, \theta} \coloneqq \abs{\kk}^{\ddim-1}
\kl{\Fo g} \kl{\kk, \theta}
\quad \text{ for } \kl{\kk, \theta} \in \R^{\ddim-1} \times \kl{0, \frac{\pi}{2}}
\,.
\end{equation}
The Riesz potential  is  well defined if
$\kl{\Fo g} \kl{\edot, \theta} \in L^1\mkl{\R^{\ddim-1}}$ for every $\theta  \in \kl{0, \pi\slash 2}$,
which will always be the case in our considerations.


\subsubsection*{Explicit reconstruction formulas}

The central results of this paper are the  following  explicit reconstruction formulas  for inverting the conical Radon transform.

\begin{theorem}[Reconstruction formulas for the conical Radon transform]
For\label{thm:fbp} every $\mm \in \R$, every $f  \in C^\infty_c \mkl{\R^{\ddim-1} \times \kl{0, \infty}}$ and
every $\kl{\xx, \yy} \in \R^{\ddim-1} \times \kl{0, \infty}$, we have
\begin{align}\label{eq:fbp}
f\kl{\xx, \yy}
&=
\frac{\yy^\mm}{\kl{2\pi}^{\ddim-1}}
	\int_{0}^{\pi/2}
	\frac{1}{\cos\kl{\theta}^{1+\mm}}
	\int_{S^{\ddim-2}}
	\kl{ \Io^{(1-\ddim)} \Ro^{(\mm)} f} \kl{\xx  + \yy \tan\kl{\theta} \win, \theta}
 	\rmd\win
 	\rmd\theta \,,
\\
\label{eq:fbp2}
f\kl{\xx, \yy}
&=
\frac{1}{\kl{2\pi}^{\ddim-1}}
	\int_{\R^{\ddim-1}}
	\frac{\kl{ \abs{\uu-\xx}^2 + \yy^2}^{\frac{\mm-1}{2}}}
	{
	\abs{\uu-\xx}^{\ddim-2}}
	\kl{ \Io^{(1-\ddim)} \Ro^{(\mm)} f}
	\kl{\uu, \arctan\kl{\frac{\abs{\uu-\xx}}{\yy}}}
	\,
	\rmd \uu
	\,.
\end{align}
Here $\Io^{(1-\ddim)}$ is the Riesz potential
 defined by \req{riesz}.
\end{theorem}

\begin{proof}
See Sections~\ref{sec:fbp} and~\ref{sec:fbp2}.
\end{proof}

The reconstruction formulas \req{fbp}, \req{fbp2}
are of the filtered back-projection type:
The Riesz potential can be interpreted as a filtration step
in the first argument and the integrations actually sum over all conical surfaces that pass through the
reconstruction point $\kl{\xx, \yy} \in \R^{\ddim-1} \times \kl{0, \infty}$. In analogy to the classical Radon transform the integration process  may therefore be called \emph{conical back-projection}.   Note that \req{fbp}, \req{fbp2} only differ up to a different parametrization of the set of all conical surfaces  passing through the reconstruction point.

For practical applications, the two and three dimensional situations are the most relevant ones. In these cases the
formulas of Theorem \ref{thm:fbp} read as follows.

\begin{corollary}[Reconstruction formulas\label{thm:fbp-23d} for  $\ddim = 2, 3$]   \mbox{}
\begin{enumerate}[topsep=0em,label=(\alph*)]
\item\label{it:fbp-2d}
Suppose $\ddim=2$. Then, for every
$f  \in C^\infty_c \skl{\R \times \kl{0, \infty}}$ and every
$ \kl{x, \yy} \in   \R  \times \kl{0, \infty}$,
\begin{align*}
f\kl{x, \yy}
&=
\frac{\yy^\mm}{2\pi}
	\int_{0}^{\pi/2}
	\frac{\mkl{\partial_u \Hi_u  \Ro^{(\mm)} f}
	\kl{x  +  \win \yy \tan\kl{\theta} , \theta}
	+
	\mkl{\partial_u \Hi_u  \Ro^{(\mm)} f}
	\kl{x  -  \win \yy \tan\kl{\theta} , \theta}}
	{\cos\kl{\theta}^{1+\mm}}
	\, \rmd\theta \,,
	\\
	f\kl{x, \yy}
		&=
		\frac{1}{2\pi}
	\int_{\R}
	\kl{\abs{u-x}^2+\yy^2}^{\frac{\mm-1}{2}}
	\kl{\partial_u \Hi_u  \Ro^{(\mm)} f}
	\kl{x , \arctan \kl{\frac{\abs{u-x}}{\yy}}}
 	\, \rmd u  \,.
\end{align*}
Here $\partial_u$ and  $\Hi_u$ denote the derivative and the Hilbert transform in first argument.

\item\label{it:fbp-3d}
Suppose  $\ddim=3$.
Then, for every
$f  \in C^\infty_c \mkl{\R^2 \times \kl{0, \infty}}$ and
every $ \kl{\xx, \yy} \in   \R^{2} \times \kl{0, \infty}$,
\begin{align*}
f\kl{\xx, \yy}
&=
- \frac{\yy^\mm}{ 4\pi ^2}
	\int_{0}^{\pi/2}
	\frac{1}{\cos\kl{\theta}^{1+\mm}}
	\int_{S^{1}}
	\kl{ \Delta_{\uu} \Ro^{(\mm)} f} \kl{\xx  + \yy \tan\kl{\theta} \win, \theta}
 	\rmd\win
 	\rmd\theta
\\
f\kl{\xx, \yy}
&=
-
\frac{1}{4\pi^2}
	\int_{\R^{2}}
	\frac{\kl{\abs{\uu-\xx}^2+\yy^2}^{\frac{\mm-1}{2}}}
	{\abs{\uu-\xx}}
	\kl{\Delta_{\uu} \Ro^{(\mm)} f}
	\kl{\uu, \arctan\kl{\frac{\abs{\uu-\xx}}{\yy}}}
	\,
	\rmd \uu
	 \,.
\end{align*}
Here $\Delta_{\uu}$ denotes the Laplacian in the first component.
\end{enumerate}
\end{corollary}

\begin{proof} \mbox{}
For  $\ddim=2$ we have  the Fourier representation
$\mkl{\Fo  \Hi_u f}  \kl{k} = - i \operatorname{sign} \kl{k} \mkl{\Fo f } \kl{k}$
and $\mkl{\Fo \partial_u f}  \kl{k} = i k \mkl{\Fo f } \kl{k}$ of the Hilbert transform and the one dimensional derivative, respectively. This  shows
$\Io^{-1}  =  \partial_u \Hi_u $. Hence Item \ref{it:fbp-2d}
follows from Theorem \ref{thm:fbp}.
Similarly,  for $\ddim=3$, we have  $\Io^{-2}  =  - \Delta_{\uu}$ and hence Item \ref{it:fbp-3d} again follows from Theorem \ref{thm:fbp}.
\end{proof}

For $\ddim=2$, formulas equivalent to the ones of Theorem~\ref{thm:fbp-23d}~\ref{it:fbp-2d} have been first derived
in \cite{BasZenGul97,TruNgu11}.
The three dimensional reconstruction formulas of Theorem~\ref{thm:fbp-23d}~\ref{it:fbp-3d}  (as well as the  higher dimensional generalizations of Theorem \ref{thm:fbp}) are  new.
One notes, that in three spatial dimensions the reconstruction formulas are  particularly
simple and further local: The reconstruction of  $f$ at  some reconstruction point $\kl{\xx, \yy}$  only requires the integrals over cones passing through an arbitrarily small neighbourhood  of  $\kl{\xx, \yy}$. Since for any odd $\ddim$, the Riesz potential satisfies $\Io^{(1-\ddim)}  =
\kl{-1}^{(\ddim-1)/2}  \Delta_{\uu}^{(\ddim-1)/2}$, the reconstruction formulas \req{fbp}, \req{fbp2}
are in fact local for every odd space dimensions. Contrary, in even space dimension \req{fbp}, \req{fbp2} are
non-local: Recovering a function at a single point requires knowledge of the integrals over all conical surfaces.
This  behaviour is similar to the one of the
classical Radon transform, where also the inversion is local in odd and non-local in even dimensions
(see, for example, \cite[p. 20]{Nat01}).

\subsubsection*{A Fourier slice identity}

Theorem~\ref{thm:fbp}  will be established using the
following Theorem \ref{thm:fourierslice}, which an analogon of the well known Fourier slice identity \cite[Chapter 1, Theorem 1.1]{Nat01} satisfied by the classical Radon transform.

 \begin{theorem}[Fourier\label{thm:fourierslice} slice identity for the conical Radon transform]
For every $\mm\in \R$, every $f \in C_c^\infty \mkl{\R^{\ddim-1} \times (0, \infty)}$
and every $ \kl{\kk,\theta } \in \R \times (0, \pi/2)$, we have
\begin{equation} \label{eq:fourierslice}
	\kl{ \Ho_{\frac{\ddim-3}{2}} \Fo
	\yy^{\frac{\ddim-3}{2} - \mm}f}
	\kl{\kk, \abs{\kk} \tan\kl{\theta}}
	=
	\kl{2\pi}^{\frac{1-\ddim}{2}}
	\frac{  \cos\kl{\theta}^{1-\mm}}{\tan\kl{\theta}^{\frac{\ddim-1}{2}}}
	\abs{\kk}^{\frac{\ddim-3}{2}}
	\kl{\Fo \Ro^{(\mm)} f}\kl{\kk,\theta}
	\,.
\end{equation}
Here  $\yy^{\skl{\ddim-3}\slash 2-\mm}f$ is the function
$\kl{\xx, \yy} \mapsto  \yy^{\skl{\ddim-3}\slash 2 -\mm}f\kl{\xx, \yy}$, $\Fo$ the Fourier transform
in the first argument, and
$\Ho_{\skl{\ddim-3}\slash 2}$
 the  Hankel transform of order $\skl{\ddim-3}\slash 2$
 in the second argument.
\end{theorem}

 \begin{proof}
See Section~\ref{sec:fs}.
\end{proof}

The Fourier slice identity is of course of interest on its own.
The argument $(\kk, \abs{\kk} \tan\theta)$, for $\kk \in \R^{\ddim-1}$
and  $\al \in (0, \pi/2)$, appearing on the left hand side of  \req{fourierslice},
fills in the whole upper half-space, which is required to invert the Fourier-Hankel transform using well known explicit and stable inversion formulas. Hence
the function $f$ can be reconstructed based on
\req{fourierslice} by means of a $\ddim-1$-dimensional Fourier
transform, followed by an interpolation, and finally performing an inverse
$\ddim$-dimensional Fourier-Hankel transform.

\subsection{Outline}

The remainder of the paper is mainly devoted to the
proofs of  Theorems~\ref{thm:fbp} and \ref{thm:fourierslice}
that we will establish in the following Section \ref{sec:proofs}. We will first derive the Fourier slice identity of Theorem \ref{thm:fourierslice}, which will then be used to proof
the reconstruction formulas of Theorem~\ref{thm:fbp}.
The paper ends with a discussion in Section \ref{sec:discussion}.

\section{Proofs of the main results}
\label{sec:proofs}

In this section  we derive  Theorems~\ref{thm:fbp} and \ref{thm:fourierslice}.
The following elementary Lemma shows that
it suffices to derive these results for the special
case $\mm=0$.

\begin{lemma}[Relation between $\Ro^{(\mm)}$ and $\Ro^{(0)}$]
For\label{lem:munu} every  $\mm \in \R$,  every
$f \in C_c^\infty \mkl{\R^{\ddim-1} \times (0, \infty)}$
and every $\kl{\kk,\theta } \in \R^{\ddim-1} \times (0, \pi/2)$,
we have
\begin{equation}\label{eq:munu}
\kl{\Ro^{(\mm)} f} \kl{\uu, \theta} =
\kl{ \cos\kl{\theta}^{\mm}
\Ro^{(0)}   \yy^{-\mm} f} \kl{\uu, \theta} \,.
\end{equation}
Here $\yy^{-\mm}$ stands for the operator that multiplies a  function
$f\kl{\xx, \yy}$ by $\yy^{-\mm}$ and likewise $\cos\kl{\theta}^{\mm}$ stands for
the operator that multiplies $g \kl{\uu, \theta}$ by $\cos\kl{\theta}^{\mm} $.
\end{lemma}

\begin{proof}
The definition of $\Ro^{(\mm)}$ and the substitution
$s = \yy \slash \cos \kl{\theta}$ yield
\begin{multline*}
\kl{\Ro^{(\mm)} f} \kl{\uu, \theta}
=
\int_0^\infty \frac{1}{s^\mm} \int_{S^{\ddim-2}}
	f \kl{\uu + s \sin\kl{\theta} \win ,  s \cos\kl{\theta} }
	\kl{s \sin\kl{\theta}}^{\ddim-2} \rmd\win \rmd s
\\
=
\cos\kl{\theta}^{\mm-1}
\int_0^\infty \int_{S^{\ddim-2}}
	\yy^{-\mm}  f \kl{\uu + \yy \tan\kl{\theta} \win , \yy }
	\kl{\yy \tan\kl{\theta}}^{\ddim-2}
\rmd\win\rmd \yy
 \,.
\end{multline*}
Comparing the last expression for $\Ro^{(\mm)} f $  with the corresponding expression for
$\Ro^{(0)} f$ obviously shows \req{munu}.
\end{proof}

\subsection{Proof of  Theorem ~\ref{thm:fourierslice}
(the Fourier slice identity)}
\label{sec:fs}

We start by showing \req{fourierslice} for the special
case $\mm =0$. The general case will then be a
consequence of Lemma~\ref{lem:munu}.

The definition of the conical Radon transform, the definition of the Fourier
transform and some basic manipulations yield
\begin{align*}
	\kl{\Fo \Ro^{(0)} f}\kl{\kk,\theta}
    & =
	\int_{\R^{\ddim-1}} e^{-i \kk\uu } \kl{\Ro^{(0)} f}\kl{\uu, \theta}
    	\rmd \uu
	\\
	& =
	\int_{0}^\infty
	\int_{S^{\ddim-2} }
	 \int_{\R^{\ddim-1}} e^{-i \kk \uu}
	 f \kl{ \uu + s \sin \kl{\theta} \win ,  s \cos\kl{\theta}}
	\kl{s \sin\kl{\theta}}^{\ddim-2}
	\,
    \rmd\uu
    \rmd\win
    \rmd s
	\\
	& =
	\int_{0}^\infty
		\kl{s \sin\kl{\theta}}^{\ddim-2}
	\int_{S^{\ddim-2} }
	 e^{ i \kk s \sin \kl{\theta} \win}
	 \int_{\R^{\ddim-1}} e^{-i \kk \uu}
	 f \kl{ \xx ,  s \cos\kl{\theta}}
	\,
    \rmd\uu
    \rmd\win
    \rmd s
		\\
	& =
	\int_{0}^\infty
	\kl{s \sin\kl{\theta}}^{\ddim-2}
	\kl{\Fo f}  \kl{\kk, s \cos\kl{\theta}  }
	\ekl{\int_{S^{\ddim-2} }
	 e^{ i \kk s \sin \kl{\theta} \win}
	\rmd\win }
	\rmd s\,.
\end{align*}
Now we use the identity (see, for example, \cite[page 198]{Nat01}),
\begin{equation} \label{eq:ps}
\int_{S^{\ddim-2}}
e^{- i \kk r  \win } \rmd\win
=
\kl{2\pi}^{\frac{\ddim-1}{2}} \abs{\kk}^{\frac{3-\ddim}{2}} r^{\frac{3-\ddim}{2}}
J_{\frac{\ddim-3}{2}} \kl{ \abs{\kk}r  }
\quad \text{ for all }
\kl{\kk, r}  \in \R^{\ddim-1} \times \kl{0, \infty}
\,.
\end{equation}
Application of \req{ps} with $r = s \sin\kl{\theta}$ followed by the substitution
$ s= \yy \slash \cos \kl{\theta}$  yields
\begin{align*}
	\kl{\Fo \Ro^{(0)} f}\kl{\kk,\theta}
	&=
	\kl{2\pi}^{\frac{\ddim-1}{2}}
	\int_{0}^\infty
	\kl{s \sin\kl{\theta}}^{\ddim-2}
	\kl{\Fo f}  \kl{\kk, s \cos\kl{\theta}  }
	\abs{\kk}^{\frac{3-\ddim}{2}} \kl{s \sin\kl{\theta}}^{\frac{3-\ddim}{2}}
J_{\frac{\ddim-3}{2}} \kl{ \abs{\kk}s \sin\kl{\theta}  }
	\rmd s
	\\
	&=
	\kl{2\pi}^{\frac{\ddim-1}{2}}
	\int_{0}^\infty
	\kl{s \sin\kl{\theta}}^{\frac{\ddim-1}{2}}
	\kl{\Fo f}  \kl{\kk, s \cos\kl{\theta}  }
	\abs{\kk}^{\frac{3-\ddim}{2}}
J_{\frac{\ddim-3}{2}} \kl{ \abs{\kk}s \sin\kl{\theta}  }
	\rmd s
	\\
	&=
	\kl{2\pi}^{\frac{\ddim-1}{2}}
	\int_{0}^\infty
	\kl{\yy \tan\kl{\theta}}^{\frac{\ddim-1}{2}}
	\kl{\Fo f}  \kl{\kk,  \yy }
	\abs{\kk}^{\frac{3-\ddim}{2}}
J_{\frac{\ddim-3}{2}} \kl{ \abs{\kk} \yy\tan\kl{\theta}  }
	\frac{\rmd \yy}{\cos\kl{\theta}}
	\\
	&=
	\kl{2\pi}^{\frac{\ddim-1}{2}}
	\frac{\tan\kl{\theta}^{\frac{\ddim-1}{2}}}{  \cos\kl{\theta}}
	\abs{\kk}^{\frac{3-\ddim}{2}}
	\int_{0}^\infty
	\yy^{\frac{\ddim-3}{2}}
	\kl{\Fo f}  \kl{\kk,  \yy }
J_{\frac{\ddim-3}{2}} \kl{ \abs{\kk} \yy\tan\kl{\theta}  }
	\yy \rmd \yy \,.
\end{align*}
The last displayed equation is recognised  as the Hankel transform of order $\kl{\ddim-3}/2$
of $\Fo f$ in the second argument. We conclude, that
 \begin{equation} \label{eq:fs-0}
	\kl{\Fo \Ro^{(0)} f}\kl{\kk,\theta}
		=
		\kl{2\pi}^{\frac{\ddim-1}{2}}
		\frac{\tan\kl{\theta}^{\frac{\ddim-1}{2}}}{  \cos\kl{\theta}}
		\abs{\kk}^{\frac{3-\ddim}{2}}
		\kl{ \Ho_{\frac{\ddim-3}{2}} \Fo \yy^{\frac{\ddim-3}{2}}f}	\kl{\kk, \abs{\kk} \tan\kl{\theta}}
	\,.
\end{equation}
This shows \req{fourierslice} for the special case
$\mm =0$.

For general $\mm \in \R$ we use the relation
$\Ro^{(0)} \yy^{-\mm} f =
\cos \kl{\theta}^{-\mm}  \Ro^{(\mm)} f$ from Lemma \ref{lem:munu}. Together with \req{fs-0}
this yields
\begin{multline*}
	\kl{ \Ho_{\frac{\ddim-3}{2}} \Fo
	\yy^{\frac{\ddim-3}{2}-\mm}f}
	\kl{\kk, \abs{\kk} \tan\kl{\theta}}
	=
	\kl{ \Ho_{\frac{\ddim-3}{2}} \Fo
	\yy^{\frac{\ddim-3}{2}} \yy^{-\mm}f}
	\kl{\kk, \abs{\kk} \tan\kl{\theta}}
	\\
	=
	\kl{2\pi}^{\frac{1-\ddim}{2}}
	\frac{\cos\kl{\theta}}{\tan\kl{\theta}^{\frac{\ddim-1}{2}}}
	\abs{\kk}^{\frac{\ddim-3}{2}}
	\kl{\Fo \Ro^{(0)} \yy^{-\mm} f}\kl{\kk,\theta}
	=
	\kl{2\pi}^{\frac{1-\ddim}{2}}
	\frac{  \cos\kl{\theta}^{1-\mm}}
	{\tan\kl{\theta}^{\frac{\ddim-1}{2}}}
	\abs{\kk}^{\frac{\ddim-3}{2}}
	\kl{\Fo \Ro^{(\mm)}f}\kl{\kk,\theta}
	\,.
\end{multline*}
This is  \req{fourierslice} for the
case of general $\mm \in \R$ and concludes the proof of
Theorem~\ref{thm:fourierslice}.

\subsection{Proof of reconstruction formula \req{fbp}}
\label{sec:fbp}

We start with the proof of \req{fbp} for $\mm=0$.
Application of the inversion formulas for the Fourier and the Hankel transform followed by  the substitution $\ky = \abs{\kk} \tan\kl{\theta}$ shows
\begin{align*}
\yy^{\frac{\ddim-3}{2}}f \kl{ \xx, \yy}
&=
\frac{1}{\kl{2\pi}^{\ddim-1}}
\int_{\R^{\ddim-1}}
\int_0^\infty
\kl{ \Ho_{\frac{\ddim-3}{2}} \Fo \yy^{\frac{\ddim-3}{2}}f}
\kl{\kk, \ky } J_{\frac{\ddim-3}{2}} \kl{\ky \yy}
e^{i\kk\xx}
 \,
\ky \rmd \ky
\, \rmd \kk
\\
&=
\frac{1}{\kl{2\pi}^{\ddim-1}}
\int_{\R^{\ddim-1}}
\int_0^{\pi/2}
\kl{ \Ho_{\frac{\ddim-3}{2}} \Fo \yy^{\frac{\ddim-3}{2}}f}
\kl{\kk, \tan\kl{\theta} \yy } J_{\frac{\ddim-3}{2}} \kl{\abs{\kk} \tan\kl{\theta} \yy}
e^{i\kk\xx}
 \frac{\abs{\kk}^2 \tan\kl{\theta}}{\cos\kl{\theta}^2}
\,
\rmd \theta
 \rmd \kk \,.
\end{align*}
Application of the
Fourier slice identity (Theorem \ref{thm:fourierslice})
with $\mm=0$ and  interchanging the order of integration then yields
\begin{equation}\label{eq:aux1}
\yy^{\frac{\ddim-3}{2}}f \kl{ \xx, \yy}
=
\kl{2\pi}^{\frac{3\kl{1-\ddim}}{2}}
\int_0^{\pi/2}
\frac{\kl{ \tan\kl{\theta}}^{\frac{3-\ddim}{2}}}{\cos\kl{\theta}}
\int_{\R^{\ddim-1}}
	\abs{\kk}^{\frac{\ddim+1}{2}}
	\kl{\Fo \Ro^{(0)} f}\kl{\kk,\theta}
 J_{\frac{\ddim-3}{2}} \kl{\abs{\kk} \tan\kl{\theta} \yy}e^{i\kk\xx}
\rmd \kk
\, \rmd \theta \,.
\end{equation}

By \req{ps}, we have
\begin{equation*}
J_{\frac{\ddim-3}{2}} \kl{\abs{\kk} \tan\kl{\theta} \yy}
=
\kl{2\pi}^{\frac{1-\ddim}{2}} \abs{\kk}^{\frac{\ddim-3}{2}} \kl{\tan\kl{\theta} \yy}^{\frac{\ddim-3}{2}}
\int_{S^{\ddim-2}}
e^{- i \kk \tan\kl{\theta} \yy \win } \rmd\win \,.
\end{equation*}

Therefore,
\begin{multline*}
\int_{\R^{\ddim-1}}
	\abs{\kk}^{\frac{\ddim+1}{2}}
	\kl{\Fo \Ro^{(0)} f}\kl{\kk,\theta}
 J_{\frac{\ddim-3}{2}} \kl{\abs{\kk} \tan\kl{\theta} \yy}e^{i\kk\xx}
\rmd \kk
\\
\begin{aligned}
&=
\kl{2\pi}^{\frac{1-\ddim}{2}}
\int_{S^{\ddim-2}}
\kl{\tan\kl{\theta} \yy}^{\frac{\ddim-3}{2}}
\ekl{
\int_{\R^{\ddim-1}}
	\abs{\kk}^{\ddim-1}
	\kl{\Fo \Ro^{(0)} f}\kl{\kk,\theta}
 e^{i \kk \kl{\xx - \tan\kl{\theta} \yy \win }}
\rmd \kk}
\\
&=
\kl{2\pi}^{\frac{\ddim-1}{2}}
\int_{S^{\ddim-2}}
\kl{\tan\kl{\theta} \yy}^{\frac{\ddim-3}{2}}
\kl{\Io^{(1-\ddim)}\Ro^{(0)} f}\kl{\xx - \tan\kl{\theta} \yy \win,\theta}
\rmd\win \,.
\end{aligned}
\end{multline*}
Together with  \req{aux1} this further implies
\begin{multline*}
\yy^{\frac{\ddim-3}{2}}f \kl{ \xx, \yy}
=
\frac{1}{\kl{2\pi}^{\ddim-1}}
\int_0^{\pi/2}
\frac{\kl{ \tan\kl{\theta}}^{\frac{3-\ddim}{2}}}{\cos\kl{\theta}}
\int_{S^{\ddim-2}}
\kl{\tan\kl{\theta} \yy}^{\frac{\ddim-3}{2}}
\kl{\Io^{(1-\ddim)}\Ro^{(0)} f}\kl{\xx - \tan\kl{\theta} \yy \win,\theta}
\rmd\win
\, \rmd \theta
\\
=
\frac{\yy^{\frac{\ddim-3}{2}}}{\kl{2\pi}^{\ddim-1}}
\int_0^{\pi/2}
\frac{1}{\cos\kl{\theta}}
\int_{S^{\ddim-2}}
\kl{\Io^{(1-\ddim)}\Ro^{(0)} f}\kl{\xx - \tan\kl{\theta} \yy \win,\theta}
\rmd\win
\, \rmd \theta
\,.
\end{multline*}
This shows formula \req{fbp} for the special case
$\mm =0$.

To show  \req{fbp}  in the general case $\mm \in \R$, we again use the relation
$\cos \kl{\theta}^{-\mm}  \Ro^{(\mm)} f =
\Ro^{(0)} \yy^{-\mm} f$ from Lemma \ref{lem:munu}. Hence application of the reconstruction formula for the special case $\mm=0$ to $\Ro^{(0)} \yy^{-\mm} f$ yields
\begin{equation*}
\yy^{-\mm}
f \kl{ \xx, \yy}
=
\kl{2\pi}^{1-\ddim}
\int_0^{\pi/2}
\frac{1}{\cos\kl{\theta}^{1+\mm}}
\int_{S^{\ddim-2}}
\kl{\Io^{(1-\ddim)} \Ro^{(\mm)} f}\kl{\xx - \tan\kl{\theta} \yy \win,\theta}
\rmd\win
\, \rmd \theta \,.
\end{equation*}
This shows \req{fbp} in the general case $\mu \in \R$.

\subsection{Proof of  reconstruction formula \req{fbp2}}
\label{sec:fbp2}

Finally we derive \req{fbp2} as an easy consequence of \req{fbp}.  To that end
we first substitute $\theta = \arctan \kl{r / \yy}$ with $r \in \kl{0, \infty}$. Then
$\rmd \theta = \yy^{-1}  \cos\kl{\theta}^2 \rmd r$
and  $\cos \kl{\theta} = 1 \slash \sqrt{1 + r^2/\yy^2}$. Consequently, \req{fbp} implies
\begin{multline*}
f \kl{\xx, \yy}
=
\frac{y^{\mm-1} }{\kl{2\pi}^{\ddim-1}}
\int_0^{\infty}
\int_{S^{\ddim-2}}
\kl{1+r^2/\yy^2}^{\frac{\mm-1}{2}}
\kl{\Io^{(1-\ddim)}\Ro^{(0)} f}\kl{\xx - r \win,
\arctan \kl{\frac{r}{ \yy}}}
\rmd\win
\, \rmd r
\\
=
\frac{1}{\kl{2\pi}^{\ddim-1}}
\int_0^{\infty}
\int_{S^{\ddim-2}}
\kl{r^2+\yy^2}^{\frac{\mm-1}{2}}
\kl{\Io^{(1-\ddim)}\Ro^{(0)} f}\kl{\xx + r \win,
\arctan \kl{\frac{r}{ \yy}}}
\rmd\win
\, \rmd r \,.
\end{multline*}
Now we substitute
$\xx + r \win =\uu$  (polar coordinates  in the plane $\R^{\ddim-1}$ around the center $\xx$).
Then $\rmd \uu = r^{\ddim-2} \rmd\win \rmd r$ and $r = \abs{\uu-\xx}$. Consequently,
\begin{equation*}
f \kl{\xx, \yy}
=
\frac{1}{\kl{2\pi}^{\ddim-1}}
\int_{\R^{\ddim-1}}
\frac{ \kl{\abs{\uu-\xx}^2+\yy^2}^{\frac{\mm-1}{2}} }{\abs{\uu-\xx}^{\ddim-2} }
\kl{\Io^{(1-\ddim)}\Ro^{(0)} f}\kl{\xx,
\arctan \kl{\frac{\abs{\uu-\xx}}{ \yy}}}
\; \rmd  \xx \,.
\end{equation*}
This is the reconstruction formula \req{fbp2}.

\section{Discussion}
\label{sec:discussion}

In this paper we derived explicit reconstruction formulas
for the conical Radon transform, which integrates a function in $\ddim$ spatial variables over all cones with vertices on a hyperplane and symmetry axis orthogonal to this plane. The derived formulas are of the back-projection type and are theoretically exact. Further, they are  local for odd $\ddim$, and non-local for even $\ddim$.
Among others, inversion of the conical Radon
transform is relevant for emission tomography using Compton cameras as proposed in \cite{EveFleTidNig77,Sin83,TodNigEve74}. Such a device measures the direction as well as the
scattering angle of an incoming photon at the front of the camera. The location of the photon emission can therefore be traced back to the surface of a cone. Recovering the density of the photon source therefore yields to the inversion of the conical Radon transform in a natural manner.

Radon transforms are the theoretical  foundation of many medical  imaging and remote sensing application. Certainly the most well known instance is the classical Radon transform, which integrates a function over hyperplanes. Among others, inversion of the classical Radon transform is important for classical transmission computed tomography and has been studied
 in many textbooks  (see, for example,  \cite{Hel99,Nat01}). Closed form reconstruction formulas are known for a long time and have first been derived already in 1917 by J. Radon \cite{Rad17}.
Another Radon transform that has been studied in detail
more recently is the spherical Radon transform.
This transform integrates a  function over spherical  surfaces (for some restricted centers of integration) and is, among others,  important for photo- and thermoacoustic tomography \cite{KucKun08}. Closed form reconstruction formulas for planar and spherical center sets have been found in
\cite{And88,Kun07a,Faw85,FinHalRak07,FinPatRak04}.  The  conical Radon transform, on the other hand,
is much less studied. In particular, closed form reconstruction formulas have only been known for the case $\ddim = 2$, see \cite{BasZenGul97,TruNgu11}.
In this paper we derived such reconstruction formulas for arbitrary dimension $\ddim  \geq 2$.
For computed tomography with Compton cameras \cite{BasZenGul98,CreBon94}, the three  dimensional case is of course the most relevant one. In this case, our reconstruction formulas have a particularly simple  structure and consist  of  an application of the Laplacian  followed by a conical back-projection. The numerical implementation seems quite straight forward following the ones of the classical or the spherical Radon transform (see, for example, \cite{FinHalRak07,Nat01}). Numerical studies, however, will  be subject of future research.

\providecommand{\noopsort}[1]{}

\end{document}